\documentclass[11pt]{amsart}

\usepackage[utf8]{inputenc}
\usepackage[margin=3cm]{geometry}
\usepackage{amsmath}
\usepackage{amssymb}
\usepackage{amsthm}
\usepackage[dvipsnames]{xcolor}
\usepackage[colorlinks=true,allcolors=NavyBlue]{hyperref}
\usepackage{fancyhdr}
\usepackage{blkarray}
\usepackage{multirow}
\usepackage{hhline}
\usepackage{xfrac,mathtools}

\newtheorem{defn0}{Definition}[section]
\newtheorem{prop0}[defn0]{Proposition}
\newtheorem{thm0}[defn0]{Theorem}
\newtheorem{lemma0}[defn0]{Lemma}
\newtheorem{corollary0}[defn0]{Corollary}
\newtheorem{example0}[defn0]{Example}
\newtheorem{remark0}[defn0]{Remark}
\newtheorem{conjecture0}[defn0]{Conjecture}
\newtheorem{code0}[defn0]{Code}

\newenvironment{definition}{\medskip \begin{defn0}}{\end{defn0}}
\newenvironment{proposition}{\medskip \begin{prop0}}{\end{prop0}}
\newenvironment{theorem}{\medskip \begin{thm0}}{\end{thm0}}
\newenvironment{lemma}{\medskip \begin{lemma0}}{\end{lemma0}}

\newenvironment{example}{\medskip \begin{example0}\rm}{\end{example0}}

\newenvironment{code}{\medskip\begin{code0}}{\end{code0}}

\newcommand{\ad}{\operatorname{adj}}
\newcommand{\tr}{\operatorname{tr}}
\newcommand{\sgn}{\operatorname{sgn}}
\newcommand{\PD}{\operatorname{PD}}
\DeclareMathOperator*{\argmax}{argmax}
\DeclareMathOperator*{\argmin}{argmin}

\numberwithin{equation}{section}

\title{Rational Maximum Likelihood Estimators of Kronecker Covariance Matrices}
\author{Mathias Drton, Alexandros Grosdos, Andrew McCormack}
\address{School of Computation, Information and Technology, Technical University of Munich, Germany 
and
Institute of Mathematics, Augsburg University, Germany} 
\email{\{mathias.drton, andrew.mccormack\}@tum.de, alexandros.grosdos@uni-a.de}
\date{\today}

\begin{document}

\maketitle

\begin{abstract}
As is the case for many curved exponential families, the computation of maximum likelihood estimates in a multivariate normal model with a Kronecker covariance structure is typically carried out with an iterative algorithm, specifically, a block-coordinate ascent algorithm. In this article we highlight a setting, specified by a coprime relationship between the sample size and dimension of the Kronecker factors, where the likelihood equations have algebraic degree one and an explicit, easy-to-evaluate rational formula for the maximum likelihood estimator can be found. A partial converse of this result is provided that shows that outside of the aforementioned special setting and for large sample sizes, examples of data sets can be constructed for which the degree of the likelihood equations is larger than one.

\smallskip
\noindent\textit{Keywords:} Gaussian distribution, Kronecker covariance; matrix normal model; maximum likelihood degree; maximum likelihood estimation. 
\end{abstract}

\section{Introduction}
\label{sec:introduction}


\subsection{Matrix normal distributions and Kronecker covariance matrices}

Matrix-variate data arises in numerous areas of science, such as medicine \cite{Shvartsman2018MatrixnormalMF,Yin2012}, economics \cite{chan2020large} and signal processing \cite{werner2008estimation}.  Matrix data is often high-dimensional, where a random matrix $Y \in \mathbb{R}^{m_1 \times m_2}$ can be viewed as a random vector in $\mathbb{R}^{m_1m_2}$ by vectorizing $Y$.  Even for modest row and column dimensions $m_1$ and $m_2$, $\text{vec}(Y)$ is a vector in a high-dimensional space, which makes it challenging to reliably infer dependencies among the entries of $Y$.  Indeed, even in a multivariate normal model that assumes $\text{vec}(Y) \sim \mathcal{N}(0,\Sigma)$ to be normally distributed with a zero mean, a total of $\binom{m_1m_2 + 1}{2}$ parameters have to be estimated in the covariance matrix $\Sigma$. A large number of independent random matrix observations from this distribution are needed to find accurate estimates of $\Sigma$. However, it is often difficult, expensive, or outright impossible to obtain such a large sample.  These high-dimensionality issues can be mitigated by imposing structural constraints on $\Sigma$, which can improve both the quality and interpretability of parameter estimates, assuming that these constraints approximately hold \cite{hoff2022core}.

For matrix data that exhibit similar correlations across rows as well as columns, a particularly useful structural constraint is to assume that $\Sigma\in\PD(m_1m_2)$ has a Kronecker product structure, i.e.,
\begin{equation}
    \label{eq:KronMatrix}
    \Sigma = \Sigma_2\otimes\Sigma_1,
\end{equation}
where $\Sigma_j \in \PD(m_j)$, $j=1,2$, and we write $\PD(m)$ for the cone of positive definite $m\times m$ matrices.  Keeping with multivariate normality, we write 
\begin{align}
\label{eqn:KronModel}
    Y \sim \mathcal{N}(0, \Sigma_2 \otimes \Sigma_1)
\end{align}
to express that the column-wise vectorization of $\text{vec}(Y)$ is a multivariate normal random vector with mean zero and covariance matrix $\Sigma_2 \otimes \Sigma_1$.  Such a distribution is also referred to as a {matrix normal distribution} for $Y$ \cite{dawid1981some}.  The matrix $\Sigma_1$ has an interpretation as the `column-covariance' and $\Sigma_2$ the `row-covariance', where the covariance between any two columns of $Y$ is proportional to $\Sigma_1$, and similarly for the rows and $\Sigma_2$. The Kronecker covariance model provides a compact representation of the covariance structure of a random matrix, having a parameter space of dimension $\binom{m_1 + 1}{2}+\binom{m_2 + 2}{2} - 1$, a number that in typical applications is far smaller than the dimension $\binom{m_1m_2 + 1}{2}$ of the unconstrained model.

From a slightly different perspective, the matrix normal distributions are the distributions obtained through bilinear transformations of a random matrix $Z \in \mathbb{R}^{m_1 \times m_2}$ that is filled with independent, standard normal random variables.  Indeed, if $A$ and $B$ are deterministic matrices of dimensions $m_1 \times m_1$ and $m_2 \times m_2$ respectively, then 
\[
AZB \;\sim\; \mathcal{N}(0, \Sigma_2 \otimes \Sigma_1) \quad\text{with}\quad \Sigma_1 = AA^{\top}, \ \Sigma_2 = BB^{\top};
\]
see also \cite{Gupta1999matrix} for further discussion of matrix normal distributions.  We note that although we focus on the matrix setting in this work, the matrix normal model can be generalized to higher-order arrays, through multilinear multiplication along the modes of a tensor \cite{Franks2021,Gerard2015}.   Without loss of generality, we consider distributions with mean matrix equal to zero.  Indeed, this assumption does not affect the generality of our later results on maximum likelihood estimation.


\subsection{Maximum likelihood estimation}

Let $Y_1,\dots,Y_n$ be a sample of independent and identically distributed random matrices in $\mathbb{R}^{m_1\times m_2}$, and consider the problem of estimating the covariance matrix of the underlying joint distribution.
Even with the dimensionality reduction of the matrix normal model, the sample sizes for matrix-variate data can be small compared to the dimension of the parameter space. It is therefore important to study the behaviour of estimators, in particular, the maximum likelihood estimator (MLE), in small sample settings. Previous work of \cite{Soloveychik2016,Drton2021,Makam2021} examines thresholds on the sample size and matrix dimensions $m_1,m_2$ needed for existence and uniqueness of the MLE of the positive definite covariance matrix $\Sigma_2\otimes \Sigma_1$ in the Kronecker covariance model from \eqref{eqn:KronModel}; Theorems 1.2 and 1.3 in \cite{Makam2021} fully determine these thresholds.  Subsequently, we refer to the estimator as the Kronecker MLE.

Building upon the growing literature on Kronecker covariance estimation, in this article we study the problem of finding the Kronecker MLE from an algebraic perspective. A focus of this work is determining the number of solutions to the rational Kronecker likelihood equations. We introduce the new notion of maximum likelihood multiplicity that refers to the number of complex solutions of the likelihood equations for a specified data set:
\begin{definition}[ML multiplicity]
\label{def:MLMultiplicity}
    Let $\nabla_\theta \ell(\theta;Y_1,\ldots,Y_n)$ be the gradient of the log-likelihood function of a model where this gradient has components that are rational functions of the parameter $\theta$. The maximum likelihood (ML) multiplicity at $Y_1,\ldots,Y_n$, $\mathrm{MLMult}(Y_1,\ldots,Y_n)$, is equal to the number of complex solutions $\theta \in \mathbb{C}^p$ to the likelihood equations  $\nabla_\theta \ell(\theta;Y_1,\ldots,Y_n) = 0$, which are sometimes also referred to as the score equations. 
\end{definition}
For the majority of statistical models the parameter space is contained in $\mathbb{R}^p$ so that the MLE is necessarily real-valued. By extending the domain of the likelihood equations from $\mathbb{R}^p$ to $\mathbb{C}^p$ techniques from algebraic geometry and computational commutative algebra can more easily be utilized. 

In the extensive literature studying the algebra of likelihood equations, a central quantity of interest is the ML degree \cite{Drton2009,sullivant:2018}, which is the ML multiplicity of a model for generic data:
\begin{definition}[ML degree]
    The maximum likelihood (ML) degree of a model is equal to $d$ if $\mathrm{MLMult}(Y_1,\ldots,Y_n) = d$ for generic $(Y_1,\ldots,Y_n)$.
\end{definition}
A property $\mathcal{P}(y)$ is said to hold generically if there exists a proper algebraic variety $V \subsetneqq \mathbb{R}^N$ such that $\mathcal{P}(y) = \texttt{True}$ for all $y \in \mathbb{R}^N \backslash V$. As proper varieties have Lebesgue measure zero, a property that holds generically also holds with probability one for any probability distribution over $\mathbb{R}^N$ that has a density with respect to Lebesgue measure.

In contrast to other work on ML degrees \cite{AmendolaLinearPSD,CoonsGaussianCov,catanese2006maximum}, the sample size $n$ plays an important role in determining the ML degree for the Kronecker likelihood. In \cite{Drton2021}, one setting where the Kronecker MLE is generically a rational function of the data, or equivalently the ML degree is one, has been identified. 




\begin{theorem}
\label{thm:m1=m2+1}
If $m_1=m_2+1$ and $n=2$ the Kronecker MLE exists uniquely and is a rational function of the sample $Y = (Y_1,Y_2)$ for generic $Y$ in $\mathbb{R}^{m_1 \times 2m_2}$.
\end{theorem}
\begin{proof}
    By \cite{ten1999simplicity}, for generic data $Y=(Y_1,Y_2)$, there exist two matrices $A(Y)\in\mathbb{R}^{m_1\times m_1}$ and $B(Y)\in\mathbb{R}^{m_2\times m_2}$ such that 
    \begin{equation}
    \label{eq:canonical}
    A(Y) Y_1 B(Y) = 
    \begin{pmatrix}
    I_{m_2}\\ 0
    \end{pmatrix}, \quad
        A(Y) Y_2 B(Y) = 
    \begin{pmatrix}
    0 \\
    I_{m_2}
    \end{pmatrix},
    \end{equation}
    where $I_{m_2}$ denotes the $m_2\times m_2$ identity matrix. Inspecting the algorithm in \cite{ten1999simplicity}, we observe that $A(Y)$ and $B(Y)$ are rational functions of $Y$.  For data in the canonical form in \eqref{eq:canonical}, the Kronecker MLE is obtained by combining the explicit formula in Proposition 7.1 of \cite{Drton2021} with their Lemma 2.3.  By Lemma 2.5 of \cite{Drton2021}, the Kronecker MLE for the original data $Y$ is seen to be a rational function of $Y$. 
\end{proof}

The main result of this paper identifies a further case where the Kronecker MLE is a rational function of the data.  The result stated now is proved as Theorem~\ref{prop:k1MLEExpression} in Section~\ref{sec:simplifying-likelihood}.

\begin{theorem}
\label{prop:k1MLEExpression-intro}
    If $m_1 +1 = n m_2$
the Kronecker MLE exists uniquely for a sample of $n$ generic matrices if and only if $n \geq m_2$, in which case it is a rational function of the sample $(Y_1,\dots,Y_n)$. 
\end{theorem}

This result is based on invariance properties of the Kronecker covariance model that allow one to simplify the maximum likelihood estimation problem (Section~\ref{sec:simplifying-likelihood}).  However, in contrast to the case from Theorem~\ref{thm:m1=m2+1}, the setting from Theorem~\ref{prop:k1MLEExpression-intro} does not allow for a single canonical form.  

To provide a partial converse to Theorem~\ref{prop:k1MLEExpression-intro}, we explore settings with $m_1+1<nm_2$ in Section~\ref{sec:degree}.  In these settings, and assuming $n$ to be sufficiently large, we are able to construct examples of data sets $(Y_1,\ldots,Y_n)$ for which $\mathrm{MLMult}(Y_1,\ldots,Y_n) > 1$. We conclude the paper with a computational study of the sample-size specific ML degrees  for the case of $m_2=2$, for which Gr\"obner basis computations are tractable.


\section{Preliminaries}
\label{sec:preliminaries}

In this section we present some preliminaries regarding the maximum likelihood estimator in the Kronecker covariance model. We show how the task of maximum likelihood estimation can be simplified by optimizing over one of the Kronecker factors and using the action of the general linear group on the set of positive definite matrices.


\subsection{Maximum likelihood estimation in the Kronecker covariance model}  



The log-likelihood function of observations $Y_1,\ldots,Y_n$ from a mean zero Gaussian distribution $\mathcal{N}(0,\Sigma)$ is equal to 
\begin{equation}
\label{eq:Gaussloglikelihood}
    \ell(K) = n \log \det K - n\tr (S K),
\end{equation}
up to scaling factors and additive constants that do not depend on $K$. In \eqref{eq:Gaussloglikelihood}, $n$ is the sample size, $S = \tfrac{1}{n}\sum_{i=1}^n Y_iY_i^\top$ is the sample covariance matrix, and $K \coloneqq \Sigma^{-1}$ is the concentration matrix. It is often convenient to parameterize the log-likelihood function in terms of $K$ rather than $\Sigma$ as the trace term is linear in the entries of $K$. Specializing the Gaussian log-likelihood to the Kronecker covariance model, let $Y_1, \dots, Y_n$  be an independent and identically distributed sample of matrices in $\mathbb{R}^{m_1 \times m_2}$ from the model $\mathcal{N}(0,\Sigma_2 \otimes \Sigma_1)$. Parameterizing the Kronecker submodel in terms of the concentration matrices $K_i := \Sigma_i^{-1}$, $i = 1,2$
the log-likelihood function \eqref{eq:Gaussloglikelihood} takes the form
\begin{equation}
\label{eq:Kroneckerloglikelihood}
        \ell(K_2 \otimes K_1) = n m_2 \log \det(K_1) + n m_1 \log \det(K_2) - \tr\left(\sum_{i=1}^nK_1Y_iK_2Y_i^{\top}\right).
\end{equation}
The Kronecker product map $(K_2,K_1) \mapsto K_2 \otimes K_1$ on the product space $\PD(m_2) \times \PD(m_1)$ is a surjective map onto the set of Kronecker covariance matrices. Each of the Kronecker factors $K_2,K_1$ are uniquely determined up to scale since $K_2 \otimes K_1 = cK_2 \otimes c^{-1}K_1$ for all $c > 0$. 


For a fixed $K_2 \in \PD(m_2)$, as a function of $K_1$ the log-likelihood $\ell(K_2 \otimes K_1)$ has the same form as the Gaussian log-likelihood \eqref{eq:Gaussloglikelihood}. The maximizer of \eqref{eq:Kroneckerloglikelihood} given $K_2$ is therefore
\begin{align}
\label{eq:ProfileMaximizer}
    \hat{K}_1(K_2) \coloneqq \underset{K_1 \in \PD(m_1)}{\argmax} \ell(K_2 \otimes K_1) =  \left(\frac{1}{nm_2}\sum_{i=1}^n Y_iK_2Y_i^{\top}\right)^{-1},
\end{align}
as long as the matrix inverse on the right-hand side exists. A necessary and sufficient condition for this inverse to exist with probability one for all $K_2 \in \PD(m_2)$  is $nm_2 \geq m_1$ \cite[Lemma 2.3]{Drton2021}. The popular flip-flop algorithm that numerically finds the Kronecker MLE is based on applying \eqref{eq:ProfileMaximizer}, along with an analogous update for $K_2$, in an iterative block-coordinate ascent procedure \cite{Dutilleul19}.

The profile likelihood function \cite{Severini} is defined as
\begin{align}
\label{eqn:ProfLikelihood}
    \ell\left(K_2 \otimes \hat{K}_1(K_2)\right) = -nm_2\log\det\left(\frac{1}{nm_2}\sum_{i=1}^n Y_iK_2Y_i^{\top}\right) + nm_1\log\det(K_2) - nm_1m_2.
\end{align}
To simplify notation, we omit additive and multiplicative constants and define
\begin{align}
\label{eqn:gDefinition}
    g(K) \coloneqq m_2\log\det\left(\sum_{i=1}^n Y_iKY_i^{\top}\right) - m_1\log\det(K  ).
\end{align}
The function $g$ is constant along any ray in $\PD(m_2)$, with $g(cK) = g(K)$ for all $c > 0$. Consequently, $g$ can viewed as a function on the subset of projective space $\mathbb{P}(\PD(m_2))$ of equivalence classes of positive definite matrices, where two matrices are equivalent if they are scalar multiples of each other.   

To find the Kronecker MLE, first $g$ can be minimized over the projective space to obtain the equivalence class 
\begin{align}
\label{eq:SimplifiedOptimization}
    [\hat{K}_2] = \underset{[K_2] \in \mathbb{P}(\PD(m_2))}{\argmin} \; g([K_2]).
\end{align}
For any representative $\hat{K}_2$ of this equivalence class, using \eqref{eq:ProfileMaximizer}, the MLE is equal to $\hat{K}_2 \otimes \hat{K}_1(\hat{K}_2)$. This expression for the MLE is independent of the choice of the representative $\hat{K}_2$ since $c\hat{K}_2 \otimes \hat{K}_1(c\hat{K}_2) = c\hat{K}_2 \otimes c^{-1}\hat{K}_1(\hat{K}_2)$. Explicit choices of representatives from each equivalence class, such as matrices $K$ with $\det(K) = 1$, can be used to parameterize the space $\mathbb{P}(\PD(m_2))$. However, it will be convenient when studying the likelihood equations associated with \eqref{eq:SimplifiedOptimization} to not impose such a parameterization.    

 Maximizing the profile likelihood is the same as sequentially maximizing the likelihood, implying that solving the optimization problem in \eqref{eq:SimplifiedOptimization} is necessary and sufficient to find the Kronecker MLE when it exists:
\begin{lemma}[Lemma 2.4 in \cite{Drton2021}]
\label{lem:ProfLikEquivalence}
When $nm_2 \geq m_1 $ the Kronecker MLE exists and is unique if and only if the profile likelihood function \eqref{eqn:ProfLikelihood} is uniquely maximized by $\hat{K}_2$  over $\mathbb{P}(\PD(m_2))$. The MLE is given by $\hat{K}_2 \otimes \hat{K}_1(\hat{K}_2)$.
\end{lemma}

Due to this lemma we focus our attention on the simplified problem of minimizing $g$ by solving its associated score equations. 

As expounded in \cite{Drton2021} and \cite{Makam2021}, there are subtle issues regarding the existence and uniqueness of the Kronecker MLE. The condition $nm_2 \geq m_1$, while sufficient for \eqref{eq:ProfileMaximizer} to be well-defined almost surely, is not a strong enough condition to ensure that the MLE exists and is unique. If $N_e(m_1,m_2)$ and $N_u(m_1,m_2)$ are the respective thresholds on the minimum sample size $n$ needed for the existence and uniqueness of the MLE, then it holds that
\begin{align}
\label{eq:SampleSizeReq}
    \max \bigg\{\frac{m_1}{m_2},\frac{m_2}{m_1} \bigg\}
    \leq N_e(m_1,m_2)
    \leq N_u(m_1,m_2)
    \leq \bigg\lfloor \frac{m_1}{m_2} + \frac{m_2}{m_1} \bigg\rfloor +1 .
\end{align}
The different cases considered in Theorem 1.2 in \cite{Makam2021} give the precise values of the thresholds.
It will be assumed throughout the remainder of this article that the sample size is large enough so that the Kronecker MLE exists and is unique. The function $g$ in \eqref{eqn:gDefinition} has the property of being geodesically convex along the affine-invariant geodesics in the positive definite cone \cite{Wiesel2012}. We note that under the assumption that the sample size is large enough for the existence and uniqueness of the MLE, $g$ is strictly geodesically convex, and there is only a single critical point of the likelihood equations that lies in $\mathbb{P}(\PD(m_2))$.    


\subsection{Group action}
\label{subsec:groupaction}
The matrix normal family is a group transformation family, where if $Y \sim \mathcal{N}(0,\Sigma_2 \otimes \Sigma_1)$ then $AYB^\top \sim \mathcal{N}(0, B\Sigma_2B^\top \otimes A\Sigma_1A^\top)$ for non-singular matrices $A,B$. This group action can be leveraged to simplify the form of the profile likelihood function. Consider the action of the general linear group $\text{GL}(m_1,\mathbb{R})$ on left on the data matrices $Y_i \in \mathbb{R}^{m_1 \times m_2}$ given by 
\[
Y_i \mapsto AY_i,\; i = 1,\dots ,n.
\]
Evaluating $g$ at these transformed matrices gives
\begin{align}
\label{eq:gtransformed}
    m_2 \log \det \left(\sum_{i=1}^{n} AY_iK Y_i^\top A^{\top}\right) - m_1 \log\det(K) = g(K) + 2\log\det(A). 
\end{align}
As the function \eqref{eq:gtransformed} differs from $g$ by an additive constant that does not depend on $K$, minimizing $g$ is equivalent to minimizing \eqref{eq:gtransformed} for any $A 
\in \text{GL}(m_1,\mathbb{R})$. 

Concatenating the columns of the matrices $Y_1,\dots,Y_n$ produces a matrix $Y = [Y_1|Y_2|\dots|Y_n]$ of size $m_1 \times nm_2$. The left group action described above acts on $Y$ by left multiplication:
\begin{align*}
    [AY_1| AY_2| \dots | AY_n] = AY.
\end{align*}
For generic data, the left-most $m_1 \times m_1$ block of $Y$, say $Y_*$, is in $\text{GL}(m_1,\mathbb{R})$. Taking $A = Y_*^{-1}$, $AY$ has the form
\begin{equation}
\label{eq:groupaction} 
    [I_{m_2} | C] =    \begin{bmatrix}
            1 & 0 & 0 & \ldots & 0 & c_{1,1} & \ldots & c_{1,nm_2-m_1}\\
            0 & 1 & 0 & \ldots & 0 & c_{2,1} & \ldots & c_{2,nm_2-m_1}\\
            0 & 0 & 1 & \ldots & 0 & c_{3,1} & \ldots & c_{3,nm_2-m_1}\\
            \vdots & \vdots & \vdots & \ddots & \vdots & \vdots & \ddots & \vdots\\
            0 & 0 & 0 & \ldots & 1 & c_{m_1,1} & \ldots & c_{m_1,nm_2-m_1}&
            \end{bmatrix}.
\end{equation} 
 for some $m_1 \times (nm_2 - m_1)$ dimensional matrix $C$. We denote the number of columns of $C$ by $k:=nm_2 - m_1$ throughout the remainder of this work. Without loss of generality, \eqref{eq:gtransformed} implies that it can be assumed that $Y$ has the simplified form  \eqref{eq:groupaction} when minimizing $g$. 
 
\section{A symbolic solution for the Kronecker MLE}
\label{sec:simplifying-likelihood}

In this section it will be demonstrated how $g$ can be further simplified when $Y$ is in the form \eqref{eq:groupaction}. We will show that when $k = 1$, the optimization problem \eqref{eq:SimplifiedOptimization} can be solved explicitly, and the Kronecker MLE has a closed-form expression. The key lemma that simplifies the function $g$ defined in \eqref{eqn:gDefinition} is:

\begin{lemma}
\label{lem:DeterminantReduction}
Let $nm_2 = m_1 + k$ for $k > 0$. If there exists a matrix $C \in \mathbb{R}^{m_1 \times k}$ such that $Y = [Y_1 | \ldots | Y_n] = [I_{m_1} | \;C] \in \mathbb{R}^{m_1 \times n m_2}$ and we define $D^\top = [C^\top | -I_k] \in \mathbb{R}^{k \times nm_2}$ then
\begin{align}
   \label{eqn:KeyLemmaDet} \det\left(\sum_{i=1}^{n}Y_iKY_i^\top\right) = \det(Y(I_n\otimes K)Y^\top) = \det(K)^n \det\left(D^\top (I_n \otimes K^{-1}) D \right).
\end{align}
\end{lemma}
\begin{proof}
Applying the Cauchy-Binet formula twice, 
we obtain
\begin{align}
    \label{eq:CauchyBinetonYKY}
    \det\left(\sum_{i=1}^{n}Y_iKY_i^\top \right) &= \det(Y(I_n\otimes K)Y^\top) \\
    &= 
    \sum_{J_1,J_2}
    \det(Y_{[m_1],J_1})\det((I_n\otimes K)_{J_1, J_2})\det(Y^\top_{J_2,[m_2]})
    \\
    \label{eqn:CBpolynomial}
    & = \sum_{J_1, J_2} A_{J_1}A_{J_2} Z_{J_1,J_2},
\end{align}
where the summation is over subsets $J_1,J_2 \subset [nm_2]$ with $|J_1|=|J_2|=m_1$, and $A_{J_1}  \coloneqq \det(Y_{[m_1],J_1}) = \det(Y^\top_{J_1,[m_2]})$, $Z_{J_1,J_2} \coloneqq \det((I_n\otimes K)_{J_1, J_2})$. A similar application of Cauchy-Binet establishes that
\begin{align}
\det\left(D^\top(I_n \otimes K^{-1})D\right) &  = \sum_{F_1,F_2}
    \det(D^\top_{[k],F_1})\det((I_n\otimes K^{-1})_{F_1, F_2})\det(D_{F_2,[k]})
    \\
    \label{eqn:CBpolynomial2}
    & = \sum_{F_1, F_2} B_{F_1}B_{F_2} W_{F_1,F_2},
\end{align}
with $B_{F} \coloneqq \det(D_{F,[k]})$, and $W_{F_1,F_2} = \det((I_n \otimes K^{-1})_{F_1,F_2})$.

We first establish the form of $\det((I_n \otimes M)_{F,G})$ for an $m_2 \times m_2$ matrix $M$, and index sets $F,G \subset [nm_2]$ with $\vert F \vert = \vert G \vert$. For $a,b\in \mathbb{Z}$, we use the notation $[a:b]$ to represent the set $[a,b] \cap \mathbb{Z}$. Let $F_i = F \cap [m_2(i - 1) + 1:m_2i] \mod(m_2 + 1)$ and similarly $G_i = G \cap [m_2(i - 1) + 1:m_2i] - m_2(i-1)$ for $i \in  \{1,\ldots,n\}$. That is, $F_i$ is the intersection of the index set $F$ with the $i$th block of size $m_2$ in $[nm_2]$, where the indices in $F_i$ are translated to take values in $[m_2]$. The quantity $\det((I_n \otimes A)_{F,G})$ is zero unless $\vert F_i \vert = \vert G_i \vert$ for all $i$. To see this, suppose without loss of generality that $\vert F_1 \vert > \vert G_1 \vert$. The first $\vert F_1 \vert$ rows of the matrix $(I_n \otimes A)_{F,G}$ have at most $\vert G_1 \vert$ columns with non-zero entries. It follows that $(I_n \otimes A)_{F,G}$ does not have full-rank, proving the assertion. Next, assume that $\vert F_i \vert = \vert G_i \vert$ for all $i$. In this case, $(I_n \otimes A)_{F,G} = \text{diag}( A_{F_1,G_1},\ldots,A_{F_n,G_n}) $ is block-diagonal with determinant $\prod_{i = 1}^n \det(A_{F_i,G_i})$. Defining $S_j$ to be the set of index pairs $(F,G)$ contained in $[nm_2]$ of size $\vert F \vert = \vert G \vert = j$ with $\vert F_i \vert = \vert G_i \vert$ for all $i$, the sums \eqref{eqn:CBpolynomial} and \eqref{eqn:CBpolynomial2} reduce to 
\begin{align}
\label{eqn:CB1}
    \sum_{J_1, J_2} A_{J_1}A_{J_2} Z_{J_1,J_2} & = \sum_{(J_1,J_2) \in S_{m_1} } A_{J_1}A_{J_2} \prod_{i = 1}^n \det(K_{J_{1i},J_{2i}}),
    \\
    \label{eqn:CB2}
     \sum_{F_1, F_2} B_{F_1}B_{F_2} W_{F_1,F_2} & = \sum_{(F_1,F_2) \in S_k} B_{F_1}B_{F_2} \prod_{i = 1}^n \det(K^{-1}_{F_{1i},F_{2i}}).
\end{align}
The result
\begin{align}
\label{eqn:DetKKInverse}
    \det(K_{F,G}) = \det(K) (-1)^{s(F) + 
 s(G)} \det(K^{-1}_{F^c,G^c}),
\end{align}
where $s(F) = \sum_{i \in F} i$, expresses the determinant of a minor of $K$ in terms of a determinant of a complementary minor of $K^{-1}$. Note that \eqref{eqn:DetKKInverse} holds for empty index sets $F,G$ ($F^c,G^c$ respectively) if $\det(K_{F,G})$ ($\det(K^{-1}_{F^c,G^c})$) is defined to be $1$. Applying this formula to \eqref{eqn:CB1} gives
\begin{align*}
      \sum_{J_1, J_2} A_{J_1}A_{J_2} Z_{J_1,J_2} & = \det(K)^n \sum_{(J_1,J_2) \in S_{m_1}} A_{J_1}A_{J_2} \prod_{i = 1}^n (-1)^{s(J_{1i}) + s(J_{2i})} \det(K^{-1}_{J_{1i}^c,J_{i2}^c})
      \\
      & =  \det(K)^n \sum_{(J_1,J_2) \in S_{m_1}} (-1)^{ s(J_1) + s(J_{2})} A_{J_1}A_{J_2} \prod_{i = 1}^n \det(K^{-1}_{J_{1i}^c,J_{i2}^c})
      \\
      & = \det(K)^n \sum_{(F_1,F_2) \in S_k} (-1)^{ s(F_{1}^c) + s(F_{2}^c)}A_{F_1^c}A_{F_2^c} \prod_{i = 1}^n \det(K^{-1}_{F_{1i},F_{i2}}).
\end{align*}
The second equality above follows from the fact that any index pair $j_{1a} \in J_{1i}$ and $j_{2b} \in J_{2i}$ corresponds to the index pairs $j_{1a} + (i-1)m_2$ and $j_{2b} + (i-1)m_2$ in $J_1$ and $J_2$, where $ (-1)^{j_{1a} + (i-1)m_2 + j_{2b} + (i-1)m_2} = (-1)^{j_{1a} + j_{2b}}$. As there are the same number of indices in $J_{1i}$ and $J_{2i}$ for all $i$, it is always possible to find such a pairing. 

The proof of the equality of \eqref{eqn:CBpolynomial} and \eqref{eqn:CBpolynomial2} will be completed if it is shown that $B_{F_1}B_{F_1} = (-1)^{s(J_1) + s(J_2)}A_{J_1}A_{J_2}$. Fix an $F \subset [nm_2]$ and take $J \coloneqq [nm_2]\backslash F$. We partition $F$ into two sets, $F' = F \cap [m_1]$ and $F'' = F \cap [m_1+1:nm_2] - m_1$ so that 
\begin{align*}
    B_F = \det\left( \begin{bmatrix}
    C_{F',[k]}
    \\
    -I_{F'',[k]}
    \end{bmatrix}
    \right).
\end{align*}
Partitioning $J$ analogously with $J' = J \cap [m_1]$, $J'' = J \cap [m_1 + 1:nm_2] - m_1$, we have
\begin{align*}
    A_J = \det\left( \begin{bmatrix}
        I_{[m_1],J'} \; C_{[m_1],J''}
    \end{bmatrix} \right).
\end{align*}
There is a correspondence between $J',J'',F',F''$ as $F' = [m_1]\backslash J'$, and $F'' = [k]\backslash J''$. Moreover, the equations
\begin{align*}
    \vert J' \vert + \vert F' \vert &= m_1,
    &
      \vert J'' \vert + \vert F'' \vert & = k,
      \\
        \vert J' \vert + \vert J'' \vert & = m_1,
        &
          \vert F' \vert + \vert F'' \vert & = k
\end{align*}
imply that $\vert F' \vert = \vert J''\vert$ and $\vert J' \vert = \vert F''\vert + m_1 - k$.

To compute the needed determinants we use wedge products. Let $J' = \{j_1',\ldots,j_\alpha'\}$, $J'' = \{j_1'',\ldots,j_\beta''\}$, $F' = \{f_1',\ldots,f_{\tau}'\}$, and $F'' = \{f_1'',\ldots,f_{\omega}''\}$, where the indices are ordered monotonically with $j_a' < j_b'$, $j_a'' < j_b''$, $f_a' < f_b'$, $f_a'' < f_b''$ for $a < b$. Take $c_{\cdot j_1''},\ldots c_{\cdot j_\alpha''}$, and $c_{f_1'\cdot}, \ldots, c_{f_\tau' \cdot }$ to be columns and rows of $C_{[m_1],J''}$ and $C_{F',[k]}$ respectively. Starting with $A_J$, if $e_{j_i'}$ are the standard basis vectors of $\mathbb{R}^{m_1}$, the coefficient of of $e_1 \wedge \cdots \wedge e_{m_1}$ that appears in 
\begin{align*}
    e_{j_1'} \wedge \cdots \wedge e_{ j_\alpha'} \wedge c_{\cdot j_1''} \wedge \cdots \wedge c_{\cdot j_\beta''} & = (-1)^{\alpha\beta}  c_{\cdot j_1''} \wedge \cdots \wedge c_{\cdot j_\beta''} \wedge   e_{j_1'} \wedge \cdots \wedge e_{j_\alpha'} 
    \\
    & =  (-1)^{\alpha\beta + \tfrac{1}{2}\beta(\beta-1)}  c_{\cdot j_\beta''} \wedge \cdots \wedge c_{\cdot j_1''} \wedge   e_{j_1'} \wedge \cdots \wedge e_{j_\alpha'} 
\end{align*}
is equal to $A_J$. We permute the vectors $c_{\cdot j_i''}$ in the wedge product so that $c_{\cdot j_1''}$ takes the position of the standard basis with the smallest index that is missing from $ e_{j_1'} \wedge \cdots \wedge e_{j_\alpha'}$, and continue this process with $c_{\cdot j_{2}''}$ replacing the basis vector with the second largest index, and so on. The sign of this permutation is $\sum_{x \in F'}(x - 1) = s(F') - \tau$, giving
\begin{align*}
  (-1)^{\alpha\beta + \tfrac{1}{2}\beta(\beta-1)}  c_{\cdot j_\beta''} \wedge \cdots \wedge c_{\cdot j_1''} \wedge   e_{j_1'} \wedge \cdots \wedge e_{j_\alpha'}   = (-1)^{\alpha\beta + \tfrac{1}{2}\beta(\beta-1) +  s(F') - \tau} \det(C_{F',J''}) \bigwedge_{i = 1}^{m_1} e_i. 
\end{align*}

The coefficient of $e_1 \wedge \cdots \wedge e_k$ that appears in 
\begin{align*}
 c_{f'_1 \cdot} \wedge \cdots \wedge c_{f'_\tau \cdot } \wedge (-e_{f''_1}) \wedge \cdots \wedge (-e_{f_\omega''}) & =    (-1)^{\omega} c_{f'_1 \cdot } \wedge \cdots \wedge c_{f'_\tau \cdot } \wedge e_{f''_1} \wedge \cdots \wedge e_{f_\omega''}
 \\
 & =  (-1)^{\omega + \tfrac{1}{2}\tau(\tau-1)} c_{f'_\tau \cdot } \wedge \cdots \wedge c_{f'_1 \cdot } \wedge e_{f''_1} \wedge \cdots \wedge e_{f_\omega''}
\end{align*}
is equal to $B_F$. We permute the row vectors $c_{f_i'}$ into the positions of $ e_{f''_1} \wedge \cdots \wedge e_{f_\omega''}$ so that any missing standard basis vectors are filled in by the $c_{f_i'\cdot}$s and the row ordering is maintained. The sign of this permutation is $\sum_{x \in J''}(x-1 - m_1) = s(J'') - \beta - m_1\beta$, giving
\begin{align*}
    (-1)^{\omega + \tfrac{1}{2}\tau(\tau-1)} c_{f'_\tau \cdot } \wedge \cdots \wedge c_{f'_1 \cdot } \wedge e_{f''_1} \wedge \cdots e_{f_\omega''} =   (-1)^{\omega + \tfrac{1}{2}\tau(\tau-1) + s(J'') - \beta - m_1\beta} \det(C_{F',J''})  \bigwedge_{i = 1}^k e_i.
\end{align*}
This shows that $A_J$ and $B_F$ both have the same factors that involve determinants of submatrices of $C$. 

It remains to show that the exponents of $-1$ in $(-1)^{s(J_1) + s(J_2)}A_{J_1}A_{J_2} $ and $B_{F_1}B_{F_2}$ agree. We have that $\tau = \beta$, $\alpha = m_1 - \beta$, and $\omega = k -\beta$. By the computations above, the exponent of $-1$ in $A_J^{-1}B_F$ is 
\begin{align*}
    -\alpha\beta & - \tfrac{1}{2}\beta(\beta-1) - s(F') + \tau + \omega + \tfrac{1}{2}\tau(\tau-1) + s(J'') - \beta - m_1\beta
    \\ 
    & =  (m_1 - \beta)\beta  + \tfrac{1}{2}\beta(\beta-1) + s(F') + \beta + k - \beta + \tfrac{1}{2}\beta(\beta-1) + s(J'') + \beta + m_1\beta \mod 2
    \\
    & =  s(F') + k + s(J'') \mod 2.
\end{align*}
Note that $s(F') + s(J') = s([m_1]) = \tfrac{1}{2}m_1(m_1+1)$, which simplifies the above expression to
\begin{align*}
    s(F') + k + s(J'')  & = k + s(J') + s(J'') + \tfrac{1}{2}m_1(m_1 + 1) \mod 2 
    \\
    & = s(J) + k + \tfrac{1}{2}m_1(m_1 + 1) \mod 2.
\end{align*}
Here $(J,F)$ can be taken to be either $(J_1,F_1)$ or $(J_2,F_2)$. Putting these two cases together gives
\begin{align*}
    (A_{J_1}A_{J_2})^{-1}B_{F_1}B_{F_2} & = (-1)^{s(J_1) + k + \tfrac{1}{2}m_1(m_1+1) + s(J_2) + k + \tfrac{1}{2}m_1(m_1+1)}
    \\
    & = (-1)^{s(J_1) + s(J_2)},
\end{align*}
completing the proof.
\end{proof}

Lemma \ref{lem:DeterminantReduction} replaces the determinant of the $m_1 \times m_1$ matrix  $Y(I_n\otimes K)Y^\top$ with the determinant of the $k \times k$ matrix $D^\top (I_n \otimes K^{-1}) D \coloneqq D^\top (I_n \otimes \Sigma)D$.  When $k$ is small the latter determinant can be easier to work with, as is shown in the following example.

\begin{example}
We illustrate Lemma \ref{lem:DeterminantReduction} by taking $m_1 = 4$, $m_2 = 2$, and $n = 3$ so that $k = 3 \cdot 2 - 4 = 2$. If 
\begin{align*}
    K =\begin{bmatrix}
        3 & 1 
        \\
        1 & 3 
    \end{bmatrix} \text{and } \; Y = \begin{bmatrix}
        1 & 0 & 0 & 0 & 1 & 2 
        \\
        0 & 1 & 0 & 0 & 3 & 4
        \\
        0 & 0 & 1 & 0 & 5 & 6
        \\
        0 & 0 & 0 & 1 & 7 & 8
    \end{bmatrix}
\end{align*}
then Lemma \ref{lem:DeterminantReduction} shows that
\begin{align*}
    \det\big( Y^\top (I_n \otimes K)Y \big) & = \det\begin{pmatrix}
        22 & 44 & 67 & 91
        \\
        44 & 102 & 155 & 211
        \\
        67 & 155 & 246 & 332
        \\
        91 & 211 & 332 & 454
    \end{pmatrix} = 16640 
    \\
    & = \det\begin{pmatrix}
        3 & 1
        \\
        1 & 3
    \end{pmatrix}^{3}
    \det\begin{pmatrix}
        22.375 & 25.875
        \\
        25.875 & 31.375
    \end{pmatrix} = \det(K)^n \det\left(D^\top (I_n \otimes K^{-1}) D \right).
\end{align*} 
\end{example}

To summarize our results up to this point, the Kronecker maximum likelihood estimation problem has been reformulated as follows.

\begin{proposition}
\label{prop:resultsummary}
   When it exists, the Kronecker MLE is equal to $\hat{K}_2 \otimes \hat{K}_1(\hat{K}_2)$ where $\hat{K}_1(\hat{K}_2)$ is given by \eqref{eq:ProfileMaximizer} and 
   \begin{align}
  \label{eqn:KSimpleEquation}
       [\hat{K}_2] = \underset{K \in \mathbb{P}(\PD(m_2))}{\argmin} \; m_2\log\det\left(D^\top (I_n \otimes K^{-1})D \right) + k\log\det(K).
   \end{align}
\end{proposition}
As matrix inversion is a well-defined bijective transformation of $\mathbb{P}(\PD(m_2))$, an equivalent formulation of \eqref{eqn:KSimpleEquation} is
   \begin{align}
  \label{eqn:SigmaSimpleEquation}
       [\hat{K}_2^{-1}] = \underset{\Sigma \in \mathbb{P}(\PD(m_2))}{\argmin} \; m_2\log\det\left(D^\top (I_n \otimes \Sigma)D \right) - k\log\det(\Sigma).
   \end{align}

Of special interest is the case where $k = 1$ so that $nm_2 = m_1 + 1$ and $D^\top(I_n \otimes K^{-1})D$ is a scalar. A closed-form expression for the Kronecker MLE exists in this setting:

\begin{theorem}
\label{prop:k1MLEExpression}
    If $m_1 +1 = n m_2$,
the Kronecker MLE exists uniquely for a sample of $n$ generic matrices if and only if $n \geq m_2$. If, as specified in Lemma \ref{lem:DeterminantReduction}, $D^\top = [d_1^\top,\ldots,d_n^\top]$, $d_i \in \mathbb{R}^{m_2}$, then 
\begin{align*}
    \hat{K}_2 = 
    \sum_{i = 1}^n d_id_i^\top
\end{align*}
is the unique minimizer of \eqref{eqn:KSimpleEquation}. Consequently, for this sample size the ML degree is one. 
\end{theorem}
\begin{proof}
Let $\mathcal{D}(m_2)$ denote the set $\{\Sigma \in \PD(m_2):\det(\Sigma) = 1\}$.
For this proof we parameterize $\mathbb{P}(\PD(m_2))$ by representatives in $\mathcal{D}(m_2)$. That is, we solve \eqref{eqn:SigmaSimpleEquation} subject to the constraint $\Sigma \in \mathcal{D}(m_2)$. The term $k\log\det(\Sigma)$ in \eqref{eqn:SigmaSimpleEquation} can therefore be ignored. 
We compute
    \begin{align*}
        \det\left(D^\top (I_n \otimes \Sigma)D \right) = \sum_{i = 1}^n d_i^\top \Sigma d_i = \tr\left( \Sigma \sum_{i = 1}^nd_id_i^\top\right).
    \end{align*}
    When $n \geq m_2$ the matrix $\Tilde{D} \coloneqq \sum_{i = 1}^nd_id_i^\top$ is positive definite for generic data. This function is strictly geodesically convex with respect to the affine-invariant geodesics $\gamma_{\Gamma,\Omega}(t) = \Omega^{1/2}(\Omega^{-1/2}\Gamma\Omega^{-1/2})^t \Omega^{1/2}$ \cite{Wiesel2012}. As the set of determinant one matrices is also geodesically convex, there is a unique minimizer to the defining equation of $\hat{K}_2$ in Proposition \ref{prop:resultsummary}. 
    
    If $n < m_2$ and $\Tilde{D}$ is singular we will show that the Kronecker MLE does not exist. Changing coordinates by defining $\Sigma^* = U\Sigma U^\top$ and $\Tilde{D}^* = U\Tilde{D}U^\top$ for an orthogonal $U$ matrix $U$, the optimization problem remains unchanged. That is, the minimizer of $\tr(\Sigma^*\Tilde{D}^*)$ over $\mathcal{D}(m_2)$ multiplied by $U^\top$ on the left and $U$ on the right is a minimizer of the original problem and conversely. Choosing a suitable $U$, it can be assumed without loss of generality that $\Tilde{D} = \text{diag}(\Tilde{D}_{11}, 0)$. By taking $\Sigma = \text{diag}(c\Tilde{D}_{11},\; \det(c\Tilde{D}_{11})^{1/(m_2 - n)}I_{m_2 - n})$, we see that $\det(\Sigma) = 1$ and $\tr(\Sigma \Tilde{D}) = c\Vert \Tilde{D}_{11} \Vert^2_F$. As $c \rightarrow 0$ the value of the objective function $\tr(\Sigma \Tilde{D})$ converges to $0$. However, as $\tr(\Sigma \Tilde{D}) > 0$ for all $\Sigma$ this infimum is not attainable, and the MLE does not exist.  

    By Lagrange multipliers, the optimal $\Sigma$ solves the equations
    \begin{align*}
        \Tilde{D} - \lambda \Sigma^{-1} & = 0,
        \\
        \det(\Sigma) & = 1,
    \end{align*}
    from which we obtain
    \begin{align*}
      \hat{K}_2 =  \hat{\Sigma}^{-1} =\det(\Tilde{D})^{-1/m_2} \Tilde{D}
    \end{align*}
    when $n \geq m_2$.
As any representative of the equivalence class $[\hat{K}_2]$ gives the MLE, the scale factor $\det(\Tilde{D})^{-1/m_2}$ can be dropped.
\end{proof}

The existence and uniqueness result in Theorem \ref{prop:k1MLEExpression} is corroborated by Theorem 1.2.2 in \cite{Makam2021}, where we note that $nm_2 = m_1 + 1$ implies that $\text{gcd}(m_1,m_2) = 1$. 
\bigskip

To summarize, the expression for the Kronecker MLE when $nm_2 = m_1 + 1$ is found by:
\begin{itemize}
    \item Forming the data matrix $Y = [Y_1|\ldots |Y_n]$ and splitting it as $Y = [Y_*| y]$ where $Y_* \in \mathbb{R}^{m_1\times m_1}$ and $y \in \mathbb{R}^{m_1}$.
    \item Computing $Y_*^{-1}y$ and defining the vector $v^\top = [y^\top Y_*^{-\top},-1] \in \mathbb{R}^{nm_2}$.
    \item Setting the MLE of the second Kronecker factor to $\hat{K}_2 = \sum_{i = 1}^n v_iv_i^\top$ where $v^\top = [v_1^\top,\ldots,v_n^\top]$. 
    \item Setting the MLE of the first Kronecker factor to $\hat{K}_1 = \big(\tfrac{1}{nm_2}\sum_{i = 1}^{n} Y_i \hat{K}_2 Y_i^\top\big)^{-1}$.
\end{itemize}
The computational complexity of finding this exact solution for the MLE is essentially the same as a single coordinate-ascent step of the flip-flop algorithm \cite{Dutilleul19}; the complexity of both computations is determined by the inversion or multiplication of $m_1 \times m_1$ matrices that are associated with $K_1$. Note that even when $k = 1$ the flip-flop algorithm will typically not provide an exact solution.  The closed-form expression for the MLE provided above is exact because it chooses the correct starting point for $K_2$ in the flip-flop algorithm.

\begin{example}\rm For illustrative purposes, we apply both the exact formula and the flip-flop algorithm to spatio-temporal data that records the amount of biomass feedstock in a $57 \times 74$ grid in the Gujarat region of India \cite{BiofuelData}. Note that as this state is not rectangular, feedstock quantities are not recorded at certain longitudes and latitudes. We focus on a  $4 \times 23$ subgrid for which feedstock observations are available for $8$ different years. After centering the data we look at the $4 \times 23$ dimensional feedstock matrices over the first six years so that $m_1 = 23$, $m_2 = 4$ and $n = 6$. Table \ref{tab:ExactFlipFlopComp} displays the exact expression for $\hat{K}_2$, normalized to have determinant one, along with the values of the same matrix that are computed for differing number of iterations of the flip-flop algorithm with a starting point of $I_{23} \otimes I_{4}$. From the table it is seen that $500$ flip-flop iterations are sufficient for the normalized $\hat{K}_2$ matrix to agree with the exact solution up to three decimal places, while $200$ flip-flop iterations are not. 
\end{example}

\begin{table}
\small
\centering
\caption{Entries of the $\hat{K}_2$ matrix computed exactly (top left block) and by using the flip-flop algorithm with (reading from left to right) $3,10,50,200$ and $500$ block-ascent iterations, respectively.}
\begin{tabular}{|cccc|cccc|cccc|}
  \hline
  1.451 & -0.183 & -0.690 & 0.738 & 0.998 & 0.008 & -0.038 & 0.022 & 1.278 & 0.608 & 0.376 & 0.457 \\ 
   -0.183 & 2.797 & 0.361 & -0.804 & 0.008 & 1.031 & 0.009 & -0.011 & 0.608 & 0.871 & 0.574 & 0.449 \\ 
   -0.690 & 0.361 & 0.722 & -0.356 & -0.038 & 0.009 & 0.970 & 0.007 & 0.376 & 0.574 & 1.606 & 0.742 \\ 
   0.738 & -0.804 & -0.356 & 1.248 & 0.022 & -0.011 & 0.007 & 1.005 & 0.457 & 0.449 & 0.742 & 1.517 \\ 
   \hline
    1.330 & 0.411 & -0.373 & 0.401 & 1.448 & -0.118 & -0.680 & 0.714 & 1.451 & -0.183 & -0.690 & 0.738 \\ 
   0.411 & 1.132 & 0.108 & 0.078 & -0.118 & 2.634 & 0.321 & -0.718 & -0.183 & 2.797 & 0.361 & -0.804 \\ 
   -0.373 & 0.108 & 0.859 & 0.053 & -0.680 & 0.321 & 0.723 & -0.334 & -0.690 & 0.361 & 0.722 & -0.356 \\ 
   0.401 & 0.078 & 0.053 & 1.228 & 0.714 & -0.718 & -0.334 & 1.227 & 0.738 & -0.804 & -0.356 & 1.248 \\ 
   \hline
\end{tabular}
\label{tab:ExactFlipFlopComp}
\end{table}

\section{Higher ML Degree Likelihood Equations}
\label{sec:degree}


Building on the $k = 1$ case considered previously, we describe the properties of the solution set to the Kronecker likelihood equations for larger $k$ in this section. First we precisely define the likelihood equations and ML multiplicity (see Definition \ref{def:MLMultiplicity}) in the Kronecker context. 

\begin{lemma}
\label{lem:LikelihoodFormula1}
Let $D$ be as in Lemma \ref{lem:DeterminantReduction}, where we decompose $D = [d_{ij}]_{i \in [n],j \in [k]}$ into a block matrix of vectors  $d_{ij} \in \mathbb{R}^{m_2}$. From the $d_{ij}$ we form the $m_2 \times m_2$ matrices $D_{ab} = \sum_{i = 1}^n d_{ia}d_{ib}^\top$. Equation \eqref{eqn:SigmaSimpleEquation} can be expressed in terms of the $D_{ab}$ as
\begin{align}
\label{eqn:LikelihoodDab}
 \underset{\Sigma \in \mathbb{P}(\PD(m_2))}{\argmin} \; m_2\log\det\left( [\tr(D_{ab}\Sigma)]_{a,b \in [k]} \right) - k\log\det(\Sigma) .
\end{align}
\end{lemma}
\begin{proof}
    The $(a,b)$ entry of the matrix $D^\top(I_n \otimes \Sigma)D$ is equal to $\tr(D_{ab}\Sigma)$.  
\end{proof}
We can take the gradient of the objective function in \eqref{eqn:LikelihoodDab}, viewed as a function over $\PD(m_2)$, which yields
\begin{align}
\label{eqn:gradFormula}
    m_2 \det\left( [\tr(D_{ab}\Sigma)]_{a,b \in [k]} \right)^{-1}\sum_{\pi \in \mathcal{S}_{k}} \sgn(\pi) \sum_{b = 1}^k \left(\prod_{a \neq b} \tr(D_{a\pi(a)} \Sigma) \right) D_{b\pi(b)}  - k\Sigma^{-1}.
\end{align}
The summation over $\mathcal{S}_k$ is a sum over permutations $\pi$ in the symmetric group. Setting this gradient equal to zero gives us the likelihood equations. As \eqref{eqn:gradFormula} is scale invariant with respect to $\Sigma$, the solution set to the likelihood equations is a projective variety. One measure of the complexity of symbolically solving the optimization problem associated with the Kronecker MLE is the number of points in this variety.

\begin{definition}
    The maximum likelihood multiplicity of the Kronecker MLE problem with respect to a given data set is the number of points in the projective variety that satisfies the system of polynomial equations 
    \begin{align}
    \label{eqn:LikEquations}
         m_2 \det(\Sigma) \sum_{\pi \in \mathcal{S}_{k}} \sgn(\pi) \sum_{b = 1}^k \left(\prod_{a \neq b} \tr(D_{a\pi(a)} \Sigma )\right) D_{b\pi(b)}  
 = k \det\left( [\tr(D_{ab}\Sigma)]_{a,b \in [k]} \right)\ad(\Sigma)
    \end{align}
    over the Zariski open subset of symmetric, complex matrices $\Sigma$ with $\det(\Sigma) \neq 0$ and \;\;\; $ \det\left( [\tr(D_{ab}\Sigma)]_{a,b \in [k]} \right) \neq 0$.
\end{definition}

\begin{example}
    To illustrate this definition in the case of $k = 1$, the equations \eqref{eqn:LikEquations} are
    \begin{align*}
        m_2 \det(\Sigma)D_{11} = k \tr(\Sigma D_{11}) \ad(\Sigma).
    \end{align*}
    As $\det(\Sigma),\tr(\Sigma D_{11}) \neq 0$, we obtain $\Sigma^{-1} \propto D_{11}$, implying that there is a unique solution to the likelihood equations and the ML degree is one.
\end{example}

In a rough sense, the value of $k$ represents the complexity of the likelihood equations, since more $\tr(D_{ab}\Sigma)$ terms appear in \eqref{eqn:LikEquations} for larger values of $k$. We now prove a partial converse to Theorem \ref{prop:k1MLEExpression}, showing that when $k$ is greater than one and the sample size is large enough, there exist data sets that have ML multiplicity greater than one. 

\begin{proposition}
\label{prop:ConverseMlGreaterOne}
    For every $m_2$ and $k > 1$,  if $n \geq m_2k - 1$ there exists a data set for which there are multiple, but finitely many, complex solutions to the likelihood equations. The ML multiplicity is greater than one for this data set. 
\end{proposition}
\begin{proof} 
The proof proceeds by making a convenient choice of the $[D_{ab}]_{a,b\in [k]}$ matrix and solving the resulting system of likelihood equations. The likelihood equations will have more than one, but only finitely many, solutions.  



We begin by showing that for a large enough $n$, any positive definite matrix $[D_{ab}]_{a,b\in [k]}$ can be realized in the likelihood equation \eqref{eqn:LikelihoodDab}. Note that this likelihood equation is scale invariant in $\Sigma$ so that 
\begin{align*}
m_2\log\det\left( [\tr(cD_{ab} c^{-1}\Sigma)]_{a,b \in [k]} \right) - k\log\det(c^{-1}\Sigma) + k\log(c^{-m_2}).
\end{align*}
After reparameterizing $\Sigma \mapsto c^{-1}\Sigma$, we find that if the matrix $[D_{ab}]_{a,b\in [k]}$ appears in the likelihood equation $c \times [D_{ab}]_{a,b\in [k]}$ appears in the likelihood equation with respect to the parameter $c^{-1}\Sigma$. As the number of solutions is invariant to this reparameterization, it suffices to show that for any positive definite matrix $A \in \mathbb{R}^{m_2k \times m_2 k}$ there exists a matrix $D$ and a $c > 0$ such that $A = c \times [D_{ab}]_{a,b\in [k]}$, where $D$ and $[D_{ab}]_{a,b\in [k]}$ are defined as in Lemma \ref{lem:LikelihoodFormula1}.

If in block form $D = [d_{ij}]_{i \in [n], j \in [k]}$, let $d_{i*}^\top = [d_{i1}^\top,\ldots,d_{ik}^\top]$. By definition, $[D_{ab}]_{a,b\in [k]} = 
\sum_{i = 1}^n d_{i*}d_{i*}^\top$. The only constraint on the $d_{i*}$ vectors is certain entries of $d_{n*}$ are fixed due to the presence of the $I_k$ block that appears in $D$ (Lemma \ref{lem:DeterminantReduction}). We set the other entries of the lower $m_2 \times k$ block of $D$ to be zero, which fixes the value of $d_{n*}$. As $A$ is (strictly) positive definite there exists a $c > 0$ such that $A - cd_{n*}d_{n*}^\top$ is positive semi-definite with rank $m_2k - 1$. By an eigendecomposition of $A - cd_{n*}d_{n*}^\top$, as long as $n \geq m_2k - 1$ there exists $d_{1*},\ldots,d_{n-1 *}$ such that $A - cd_{n*}d_{n*}^\top = c\sum_{i < n}d_{i*}d_{i*}^\top$. This proves the desired result that $A = c\times [D_{ab}]_{a,b\in [k]}$ for some $c$ and $D$ as long as $n \geq m_2k - 1$.   

We break the choice of $[D_{ab}]_{a,b\in [k]}$ into two cases. The first case has $m_2 > 2$ where we choose 
\begin{align*}
        D_{11} = 2I_{m_2},\;  D_{12} =e_1e_2^\top,\; D_{22} = 2I_{m_2} + e_1e_1^\top + e_2e_2^\top, \; D_{ii} = I_{m_2} \;\text{for} \  i > 2,\; D_{ij} = 0 
    \; \text{for} \ i \neq j,
\end{align*}
with $e_i$ being the $i$th standard basis vector.  The second case has $m_2 = 2$ and we take 
\begin{align*}
    D_{11} = 2I_{m_2} ,  D_{12} =e_1e_2^\top, D_{22} = 2I_{m_2} + e_1e_1^\top, \; \; D_{ii} = I_{m_2} \;\text{for} \  i > 2,\; D_{ij} = 0 \; \text{for} \ i \neq j.
\end{align*}
The likelihood equation \eqref{eqn:gradFormula} implies that any solution $\Sigma^{-1}$ is in the span of the matrices $D_{ij}$. Focusing on case one, any solution $\Sigma^{-1}$ has the form $\alpha I_{m_2} + \beta (e_1e_1^\top + e_2e_2^\top) + \gamma (e_1e_2^\top + e_2e_1^\top)$ for some coefficients $\alpha,\beta,\gamma$. Consequently, any solution $\Sigma$ is block-diagonal with the form 
\[\Sigma = \text{diag}(\begin{bmatrix}
    a & b
    \\
    b & a
\end{bmatrix},c\ldots,c).
\]
This form for $\Sigma$ could be substituted into equation \eqref{eqn:LikEquations} and solved for $a,b,c$. Instead, for simplicity we substitute this form for $\Sigma$ back into the original likelihood \eqref{eqn:LikelihoodDab} and take derivatives with respect to $a,b,c$. 

 A justification that all critical points of the likelihood are preserved by this procedure is as follows.  We can reparameterize $\Sigma$ by an invertible linear map $\phi$ on the vector space of $m_2 \times m_2$ symmetric matrices, where the first three elements of $\phi(\Sigma)$ are $(\sigma_{11},\sigma_{12},\sigma_{33}) = (a,b,c)$. The remaining elements of $\phi(\Sigma)$ are the off-diagonal entries of $\Sigma$, $\sigma_{ij}$, $i < j$, $(i,j) \neq (1,2)$ along with the diagonal entry differences $\sigma_{11} - \sigma_{22}$, $\sigma_{33} - \sigma_{jj},$ $j > 3$. We denote these remaining elements by $y$ under the reparameterization by $\phi$ so that $\phi(\Sigma) = (a,b,c,y)$. As critical points are invariant under the reparameterization, finding the critical points of $\ell(\Sigma)$ is the same as finding the critical points of $\ell(\phi^{-1}(a,b,c,y))$. It has been established above that a necessary condition for $\nabla(\ell \circ \phi^{-1})(a,b,c,y) = 0$ is $y = 0$. From the expression for the likelihood gradient \eqref{eqn:gradFormula} it is seen that $\nabla_y(\ell \circ \phi^{-1})(a,b,c,0) = 0$ for all $a,b,c$. To find all of the critical points it therefore suffices to solve the system $\nabla_{(a,b,c)}(\ell\circ\phi^{-1})(a,b,c,0) = 0$.

 In solving $\nabla_{(a,b,c)}(\ell\circ\phi^{-1})(a,b,c,0) = 0$ there remains a scale indeterminacy in the entries of $\Sigma$, or equivalently $(a,b,c)$. Without loss of generality we assume that $a$ is equal to $1$.  It will be convenient to reparameterize $(b,c)$ further by defining $t = 2 + (m_2-2)c$ to be the trace of $\Sigma$. The likelihood appearing in \eqref{eqn:LikelihoodDab} as a function of $(b,t)$ is   
\begin{align*}
  \Tilde{\ell}(b,t) \coloneqq  m_2\log\big( (4t(t+1) - b^2)t^{k-2}\big) - k\log\big((1 - b^2)(m_2-2)^{2 - m_2}(t-2)^{m_2-2}\big).
\end{align*}
It remains to solve the system of equations $\nabla\Tilde{\ell}(b,t) = 0$:
\begin{align*}
    \frac{-2m_2b}{4t(t+1) - b^2} + \frac{2kb}{1-b^2} = 0,
    \\
    m_2\left(\frac{8t + 4}{4t(t+1) - b^2} + \frac{k-2}{t} \right) - \frac{k(m_2-2)}{t-2} = 0.
\end{align*}
One solution to the first equation is $b = 0$. Setting $b = 0$ in the second equation gives
\begin{align*}
    m_2\left( \frac{k-1}{t} + \frac{1}{t+1}\right) - \frac{k(m_2-2)}{t-2} = 0,
\end{align*}
which simplifies to 
\begin{align*}
    \frac{2kt^2 +(2k - m_2-2km_2)t + 2m_2 - 2km_2}{t(t+1)(t-2)} = 0.
\end{align*}
This rational equation has two roots counting multiplicity and it is easily checked that for $m_2 > 2, k \geq 2$ the numerator is non-zero at the points $t = 0,-1,2$ where the denominator vanishes. Gr\"obner basis calculations in Macaulay 2 which are provided in the supplementary material show that the system $\nabla\Tilde{\ell}(b,t) = 0$ has at most five solutions.

The second case is handled similarly, except that there are no longer any equality or zero constraints on the entries of $\Sigma  = \begin{bmatrix}
    a & b
    \\
    b & c
\end{bmatrix}$. Setting $a = 1$ gives the likelihood
\begin{align*}
    m_2\log\big( (2(c+1)(2c+3) - b^2)(c+1)^{k-2}\big) - k\log(c - b^2).
\end{align*}
As before, $b = 0$ solves the likelihood equation for $b$. Substituting this into the remaining equation for the partial derivative of $c$ yields 
\begin{align*}
    m_2\left( \frac{k-1}{c+1} + \frac{2}{2c+3} \right) - \frac{k}{c} = \frac{(2km_2 - 2k)c^2 + (3km_2 - m_2 - 5k)c - 3k}{(c+1)(2c+3)c} = 0.
\end{align*}
The numerator of this equation has two roots counting multiplicity, neither of which makes the denominator vanish. Gr\"obner basis calculations show that there are at most four solutions to the likelihood equations, as needed. 
\end{proof}

\section{Conclusion and Open Questions}

In this paper we show that the Kronecker covariance MLE exists as a rational function of the data under a coprime relation between the sample size and the dimensions of the Kronecker factors. The central question for future work is whether the two cases identified in Theorem \ref{thm:m1=m2+1}  and Theorem \ref{prop:k1MLEExpression-intro} are all cases for which the MLE is a rational function of the data. A tenuous connection between both of the known cases where the MLE is a rational function of the data is that the $m_1 \times m_2 \times n$ tensor $Y$ of observations with $Y_{ijl} \coloneqq [Y_l]_{ij}$ has a simple canonical form. If $k = 1$ this corresponds to the form described in \cite{murakami1998case}, while if $n = 2$ this corresponds to the Kronecker canonical form described in \cite{ten1999simplicity,landsberg2011tensors}.

Proposition \ref{prop:ConverseMlGreaterOne} provides a partial description of the set of sample sizes and array dimensions where there exists a data set that has an ML multiplicity that is larger than one. The sample size condition $n \geq m_2 k-1$ is used in the proof to construct $D_{ab}$ matrices that have a tractable form. For smaller sample sizes the $-I_k$ block that appears in $D$ can no longer be ignored and the situation is more delicate, as illustrated by the setting in Theorem \ref{prop:k1MLEExpression-intro} that has $n = 2$ and $k = m_2 - 1 > 1$. 

While Proposition \ref{prop:ConverseMlGreaterOne} only makes a statement about ML multiplicity we hypothesize that this proposition also shows that the ML degree is larger than one. Conventional wisdom dictates if the ML multiplicity is larger than one for a single data set, this implies that the ML multiplicity is larger than one for generic data sets, since heuristically the degrees of polynomials in an equation system can only drop outside of a generic set (a non-empty, Zariski open set). For example, the equation $dx^2 + 2x + 1 = 0$ has two complex solutions in $x$, counting multiplicity, as long as $d \neq 0$, but has only one solution when $d = 0$. Thus the degree has dropped outside of the generic set $\mathbb{R}\backslash \{0\}$.  However, this wisdom does not hold in general, and such genericity statements depend on the structure of the likelihood equations. A simple counter-example is that the ideal $I = \langle x(x-d_1),x(x-d_2)\rangle \subset \mathbb{C}[x]$, which depends on the data $[d_1,d_2]^\top \in \mathbb{R}^2$, has only one point in the variety $V(I) = \{0\}$  generically, but has two points in $V(I) = \{0,d_1\}$ when $d_1 = d_2$. A direction of future research is to investigate general conditions on the likelihood equations that ensure that the ML multiplicity only drops outside of a generic set. It would then suffice to find a single data set with ML multiplicity larger than one to show that the ML degree is also larger than one.


\begin{table}[t]
\centering
\caption{Maximum likelihood degrees for the model $m_2=2$.}
\begin{tabular}{ |c|c|c|c|c|c|c|c|c|c|c|c|c| } 
 \hline
$m_1$ \textbackslash $n$ & 1 & 2 & 3 & 4 & 5 & 6 & 7 & 8 & 9 & 10\\ 
 \hline
 2 &   &  & 3 & 3 & 3 & 3 & 3 & 3 & 3 & 3\\ 
 3 & 0 & 1 & 4 & 7 & 7 & 7 & 7 & 7 & 7 & 7 \\ 
 4 & 0 &  & 3 & 9 & 13 & 13 & 13 & 13 & 13 & 13 \\
 5 & 0 & 0 & 1 & 7 & 16 & 21 & 21 & 21 & 21 & 21 \\
 6 & 0 & 0 &  & 3 & 13 & 25 & 31 & 31 & 31 & 31 \\
 7 & 0 & 0 & 0 & 1 & 7 & 21 & 36 & 43 & 43 & 43 \\
 8 & 0 & 0 & 0 &  & 3 & 13 & 31 & 49 & 57 & 57 \\
 9 & 0 & 0 & 0 & 0 & 1 & 7 & 21 & 43 & 64 & 73 \\
 \hline
\end{tabular}
\label{table:MLD}
\end{table}

To conclude, Table \ref{table:MLD} displays the ML degrees in the model with $m_2 = 2$ as a function of $n$ and $m_1$, computed from randomly chosen data sets. Empty positions in the table correspond to cases where there are infinitely many solutions to the system of equations. Notice that the ML degree is one when $k = 1$ in the table and that the ML degree of two does not appear. Inspecting the table, one observes that the ML degree grows and then stabilizes at the value $m_1^2 - m_1 + 1$.  For instance, $7^2 - 7 + 1 = 43$ appears in the row corresponding to $m_1 = 7$ for $n\ge 8$.   Indeed, it is apparent that the ML degree for fixed $(m_1,m_2)$ does not change once $n\ge m_1m_2$ as for such sample sizes the sample covariance matrix is a generic positive definite matrix. However, the stabilization of the ML degree occurs for smaller values of $n$ and it is an open problem to determine at what sample size the ML degree stabilizes to a limiting value, and to determine an expression for the limiting value in terms of $(m_1,m_2)$.

\appendix
\section{Macaulay 2 Code}


\begin{code}
Code for computing Gr\"obner bases of the likelihood equation ideals in Proposition \ref{prop:ConverseMlGreaterOne}.
\begin{verbatim}
-- Code for case one. 

R = QQ[y,t,b,c,m2,k, MonomialOrder => Lex];
e1 = -2*m2*b/(4*t*(t+1) - b^2) + 2*k*b/(1-b^2);
e2 = m2*((8*t+4)/(4*t*(t+1) - b^2) + (k-2)/t) - k*(m2-2)/(t-2);
I1 = ideal(numerator(e1), numerator(e2),
     y*denominator(e1)*denominator(e2) - 1);
M1 = gens gb I1;
-- We obtain a quintic in b.  
M1_0
M1_1
-- It remains to check that no values of (m2, k) cause the 
-- polynomial M1_0 to vanish entirely.
II1 = ideal(4*m2^2*k^2 + 4*m2^2*k + m2^2 - 20*m2*k^2 - 10*m2*k + 25*k^2, 
    -8*m2^2*k^2 - 8*m2^2*k - 2*m2^2 - 8*m2*k^2 + 45*m2*k + 21*k^2,
    4*m2^2*k^2 + 4*m2^2*k + m2^2 + 28*m2*k^2 - 35*m2*k);
gens gb II1

-- Code for case two.

e3 = m2*((8*c + 10)/(2*(c+1)*(2*c + 3) - b^2) + (k-2)/(c+1)) - k/(c-b^2)
e4 = -2*b*m2/(2*(c+1)*(2*c + 3) - b^2) + 2*k*b/(c-b^2)
I2  = ideal(numerator(e3), numerator(e4),
      y*denominator(e3)*denominator(e4) - 1);
M2 = gens gb I2;
-- We obtain a quartic in c.
M2_0
M2_1
II2 = ideal(8*m2^2*k^2 - 24*m2*k^2 + 16*k^2,
    30*m2^2*k^2 - 6*m2^2*k - 96*m2*k^2 + 10*m2*k + 74*k^2,
    39*m2^2*k^2 - 18*m2^2*k + m2^2 - 138*m2*k^2 + 28*m2*k + 127*k^2,
    18*m2^2*k^2 - 15*m2^2*k + 3*m2^2 - 84*m2*k^2 + 27*m2*k + 96*k^2);
gens gb II2
\end{verbatim}
\end{code}

\begin{code}
Code for finding the ML degrees in Table \ref{table:MLD}.
   \begin{verbatim}
--This function finds the ML Degree for m2 = 2
MLD = (m1,n) -> (
    m2 = 2;
    R = QQ[k_(1,2),k_(2,2),MonomialOrder=>Lex];
    for i from 1 to n do
        Y_i = matrix for j1 from 1 to m1 list for j2 from 1 to m2 list 
            sub(random(17),R);
    K = matrix {{1,k_(1,2)},{k_(1,2),k_(2,2)}};
    g1 = det sum for i from 1 to n list (Y_i)*K*(transpose Y_i);
    g2 = det K;
    I0 = ideal for e in {k_(2,2),k_(1,2)} 
                  list  m2*g2*diff(e,g1)-m1*g1*diff(e,g2);
    I = saturate(I0, g1*g2*k_(2,2));
    d = if (dim I == 0) then degree I else 0;
    return (d)
)
--We use the function to generate Table 2.
--Running time is approximately 20 minutes on an Apple M2 processor.
MLDmatrix = matrix for m1 from 2 to 9 list for n from 1 to 10 list MLD(m1,n)
   \end{verbatim}
\end{code}

\section*{Acknowledgements}
This project was supported by the European Research Council (ERC) under the European Union’s Horizon 2020 research and innovation programme (Grant agreement No.~883818). The authors wish to thank Orlando Marigliano for helpful discussions regarding ML degree.

\bibliographystyle{alpha}
\bibliography{bibliographyK}

\end{document}